\documentclass[12pt,oneside]{amsart}
\usepackage{amssymb,verbatim,enumerate,ifthen}
\usepackage{cite}
\usepackage[mathscr]{eucal}
\usepackage[utf8]{inputenc}
\usepackage[T1]{fontenc}
\usepackage{esint}
\usepackage[marginparwidth=2.4cm]{geometry}
\newtheorem{thm}{Theorem}[section]
\newtheorem{cor}[thm]{Corollary}

\newtheorem{lem}[thm]{Lemma}

\theoremstyle{definition}
\newtheorem{defin}[thm]{Definition}

\numberwithin{equation}{section}
\def\eq#1{{\rm(\ref{#1})}}
\def\Eq#1#2{\ifthenelse{\equal{#1}{*}}
  {\begin{equation*}\begin{aligned}[]#2\end{aligned}\end{equation*}}
  {\begin{equation}\begin{aligned}[]\label{#1}#2\end{aligned}\end{equation}}}

\def\A{\mathscr{A}}

\def\M{\mathscr{M}}
\def\Nm{\mathscr{N}}
\def\P{\mathscr{P}}
\def\Cts{\mathcal{C}^{2\#}}
\def\C{\mathcal{C}}
\def\S{\mathcal{S}}
\def\calM{\mathcal{M}}
\def\calQ{\mathcal{Q}}

\newcommand\R{\mathbb{R}}

\newcommand\N{\mathbb{N}}

\newcommand\Cyc{\textrm{Cyc}}
\newcommand{\QA}[1]{\A_{#1}}

\DeclareMathOperator{\conv}{conv}
\DeclareMathOperator{\conc}{conc}

\title
{On the Jensen convex and Jensen concave envelopes of means}

\author{Zsolt P\'ales}
\address{Institute of Mathematics, University of Debrecen, Pf.\ 400, 4002 Debrecen, Hungary}
\email{pales@science.unideb.hu}

\author{Pawe\l{} Pasteczka}
\address{Institute of Mathematics, Pedagogical University of Krak\'ow,  Podchor\k{a}\.{z}ych str 2, 30-084 Krak\'ow, Poland}
\email{pawel.pasteczka@up.krakow.pl}

\thanks{The research of the first author was supported by the K-134191 NKFIH Grant and by the 2019-2.1.11-T\'ET-2019-00049, EFOP-3.6.1-16-2016-00022, EFOP-3.6.2-16-2017-00015 projects. The last two projects are co-financed by the European Union and the European Social Fund.}



\begin{document}
\begin{abstract}
In recent papers the convexity of quasiarithmetic means was characterized under twice differentiability assumptions. One of the main goals of this paper is to show that the convexity or concavity of a quasiarithmetic mean implies the the twice continuous differentiability of its generator. As a consequence of this result, we can characterize those quasiarithmetic means which admit a lower convex and upper concave quasiarithmetic envelop.
\end{abstract}
\maketitle

\section{Introduction}

Jensen convex and Jensen concave means are two narrow families which play an important role in the investigation of inequalities involving means, especially the Ingham--Jessen property. Recall that two means $\M \colon I^m \to I$ and $\Nm \colon I^n \to I$ (where $I$ stands for an arbitrary interval) form an Ingham--Jessen pair if 
\Eq{*}{
&\Nm \Big( \M(x_{11},x_{12},\dots,x_{1m}), \M(x_{21},x_{22},\dots,x_{2m}),\dots,\M(x_{n1},x_{n2},\dots,x_{nm}) \Big) \\
&\qquad\le \M \Big( \Nm(x_{11},x_{21},\dots,x_{n1}), \Nm(x_{12},x_{22},\dots,x_{n2}),\dots,\Nm(x_{1m},x_{2m},\dots,x_{nm}) \Big)
}
for every matrix $x \in I^{n\times m}$. Whenever $\M,\,\Nm\colon \bigcup_{n=1}^\infty I^n\to I$ are both symmetric and repetition-invariant means such that $(\M,\Nm)$ is an Ingham--Jessen pair (for all $m,\,n\in\N$) then we can derive several interesting inequalities, among others the mixed-means inequality \cite{CarMeaNel71,Sad06} and the Kedlaya inequality \cite{Ked94} (see also \cite{PalPas16})
\Eq{*}{
\Nm \Big( x_1, \M(x_1,x_2),\dots,\M(x_1,x_2,\dots,x_n) \Big) 
\le \M \Big( x_1,\Nm(x_1,x_2), \dots,\Nm(x_1,x_2,\dots,x_n) \Big)
}
which is valid for all $n \in\N$ and $x \in I^n$; see also the recent paper \cite{ChuPalPas19} for more examples.

In the simplest case when one of the means is the arithmetic mean (from now on denoted by $\A$), we easily obtain: 
\begin{enumerate}[(i)]
 \item $(\M,\A)$ is an Ingham--Jessen pair if and only if $\M$ is Jensen concave;
 \item $(\A,\Nm)$ is an Ingham--Jessen pair if and only if $\Nm$ is Jensen convex.
\end{enumerate}
Let us stress that, due to Bernstein--Doetsch theorem \cite{BerDoe15}, in the family of means Jensen convexity and Jensen convexity coincide with convexity and concavity, respectively.

Following the ideas of convex embeddings (hulls, cones and so on) there arises a natural problem: How can we associate a convex (or concave) mean to a given one? A quite similar and comprehensive study related to the homogeneity axiom has been presented recently by the authors \cite{PalPas19a}. 

The rest of this paper is split into two parts -- we consider convex and concave envelopes in the abstract setting (section~\ref{sec:aa}) and in the quasiarithmetic setting (section~\ref{sec:qa}). 

\medskip
Let us now recall several elementary facts for the family of quasiarithmetic means.
This family was axiomatized in 1930s \cite{Kol30,Nag30,Def31}. For a continuous and strictly monotone function $f \colon I \to \R$, we define the \emph{quasiarithmetic mean} $\QA{f} \colon \bigcup_{n=1}^\infty I^n \to I$ by
\Eq{*}{
\QA{f}(a):=f^{-1} \bigg( \frac{f(a_1)+\cdots+f(a_n)}n \bigg)\qquad \text{ where }n \in \N \text{ and }a=(a_1,\dots,a_n) \in I^n.
}
The family of all quasiarithmetic means on $I$ will be denoted by $\calQ(I)$. It was Knopp \cite{Kno28} who noticed that for $I=\R_+$ and $\pi_p(x):=x^p$ ($p\ne 0$) and $\pi_0(x):=\ln x$, the quasiarithmetic mean $\QA{\pi_p}$ coincides with the $p$-th power mean $\P_p$.

An important subclass of this family (which contains power means) consists of the means which are generated by $\C^2$ functions with a nowhere vanishing first derivative -- this family of generating functions is denoted by $\Cts$ or, more frequently, $\Cts(I)$ if it is necessary to emphasize the domain. Indeed, in view of Jensen's inequality one can easily show that for $f,g \in \Cts$ the comparison inequality $\QA{f}\le\QA{g}$ is equivalent to the inequality $\tfrac{f''}{f'}\le\tfrac{g''}{g'}$ which is the comparability of two single-variable functions. For a detailed discussion concerning this relationship, we refer the reader to the paper by Mikusi\'nski \cite{Mik48}. The operator $f\mapsto\tfrac{f''}{f'}$ was used in several contexts by Pasteczka \cite{Pas13,Pas15d,Pas18c,Pas20a}.

There are few approaches to convexity (or concavity) within this family. First, in late 1980s P\'ales \cite{Pal88a} characterized the convexity of so-called quasideviation means (this family contains quasiarithmetic means). Later, the convexity of quasiarithmetic means was characterized by the authors \cite{PalPas18a} under the assumption that the mean is generated by a function from $\Cts$. The purpose of this paper is to prove that whenever a quasiarithmetic mean is convex (or concave) then its generator must belong to the class $\Cts$ -- see Theorem~\ref{thm:QAconvex} below. An extensive discussion concerning convexity and concavity in the weighted setting has been given recently by Chudziak et al. \cite{ChuGlaJarJar19}.

\section{\label{sec:aa} Abstract approach to envelopes}

In this section we prove a few preliminary results concerning convex and concave envelopes of means. Let us first introduce a formal definition of these operators.
\begin{defin}
Let $\S$ be a set of real-valued functions which are defined on a convex set $D$ of a linear space $X$. For a given function $f \colon D \to I$, we define its \emph{$\S$-convex (resp. $\S$-concave) envelopes} $\conv_\S(f)\colon D \to [-\infty,+\infty)$ and $\conc_\S(f) \colon D \to I \cup (-\infty,+\infty]$ by
\Eq{*}{
\conv_\S(f)(x)&:=\sup \{g(x) \colon g \in \S, \ g \text{ is a Jensen convex function, and }g \le f \},\\
\conc_\S(f)(x)&:=\inf \{g(x) \colon g \in \S,\, g \text{ is a Jensen concave function, and }g \ge f \}.
}
\end{defin}

Let us emphasize a few simple however important remarks. First, $\conv_\S(f)$ (resp. $\conc_\S(f)$) is either finite everywhere or $\conv_\S(f)\equiv-\infty$ (resp. $\conc_\S(f)\equiv+\infty$). 
Second, $\conv_\S(f)$ is a Jensen convex function (unless $\conv_\S(f) \equiv -\infty$) and $\conc_\S(f)$ is a Jensen concave function (unless $\conc_\S(f) \equiv +\infty$). 
Third, these operators are monotone function of both $f$ (with the pointwise ordering) and $\S$ (with the inclusion ordering). Forth, for every $f$ and $\S$ like above the inequality $\conv_\S(f) \le f \le \conc_\S(f)$ holds.
Fifth, a function $f \in \S$ is Jensen convex (resp.\ Jensen concave) if and only if $\conv_\S(f)= f$ (resp. $\conc_\S(f)= f$). 

\subsection{Evelopes in a family of means}

Throughout this section let $n \in \N$ and let $I \subset \R$ be an interval. An \emph{$n$-variable mean on $I$} is a function $\M \colon I^n \to I$ such that
\Eq{*}{
\min(x_1,\dots,x_n)\le \M(x_1,\dots,x_n) \le \max(x_1,\dots,x_n) \quad\text{ for all }x_1,\dots,x_n \in I.
}
Family of all such functions will be denoted by $\calM_n(I)$. From now on $\A_n$ stands for the $n$-variable arithmetic mean (on $I$).

We say that a family $\S \subseteq \calM_n(I)$ is \emph{permutation-closed} if, for every $\M \in \S$ and every permutation $\sigma$ of $\{1,\dots,n\}$, the mean $\M_\sigma\colon I^n \to I$ given by
\Eq{*}{
\M_{\sigma}(x_1,\dots,x_n):= \M(x_{\sigma(1)},\dots,x_{\sigma(n)})
\quad (x_1,\dots,x_n \in I)}
also belongs to $\S$. We say that $\M\colon I^n\to I$ is \emph{cyclically symmetric} (resp.\ \emph{symmetric}) if $\M=\M_\sigma$ for all cyclic permutation (resp.\ for all permutation) $\sigma$ of $\{1,\dots,n\}$. 

\begin{thm}\label{thm:M-A}
Let $\S\subseteq\calM_n(I)$ with $\A_n \in \S$. A cyclically symmetric mean $\M \in \S$ admits a Jensen convex (resp. Jensen concave) envelope in $\S$ if and only if $\M \ge \A_n$ (resp. $\M \le \A_n$). 
\end{thm}

\begin{proof} 
 First assume that there exists a Jensen convex mean $\Nm \in \S$ such that $\Nm \le \M$. 
 If we apply this inequality to all cyclic permutations of a fixed vector $x=(x_1,\dots,x_n)\in I^n$, we obtain
 \Eq{*}{
\frac1n  \sum_{\sigma \in \Cyc_n} \M(x_{\sigma(1)},\dots,x_{\sigma(n)})
 &\ge \frac1n \sum_{\sigma \in \Cyc_n} \Nm(x_{\sigma(1)},\dots,x_{\sigma(n)})
 }
As $\M$ is cyclically symmetric, the left hand side of the above inequality equals $\M(x)$. On the other hand, by the Jensen convexity of $\Nm$, we have
\Eq{*}{
\frac1n \sum_{\sigma \in \Cyc_n} \Nm(x_{\sigma(1)},\dots,x_{\sigma(n)})
&\ge \Nm\bigg(\frac1n \sum_{\sigma \in \Cyc_n} x_{\sigma(1)},\dots,\frac1n \sum_{\sigma \in \Cyc_n} x_{\sigma(n)}\bigg)\\
&= \Nm\bigg(\frac{x_1+\dots+x_n}n,\dots,\frac{x_1+\dots+x_n}n\bigg)=\frac{x_1+\dots+x_n}n.
}
This implies $\M(x)\ge\A_n(x)$. As $x$ was taken arbitrarily, we have $\M\ge\A_n$.

Conversely, if $\M  \ge \A_n$ then, as $\A_n$ is Jensen convex, we have $\M \ge \conv_\S(\M)\ge \A_n$. 

The proof of the second (i.e. concave) counterpart of this theorem is analogous.
\end{proof}

\begin{lem}
Let $\S \subset \calM_n(I)$ be a permutation-closed family.
 If $\M \in \S$ is symmetric with respect to some permutation $\sigma$ (that is $\M_\sigma=\M$) then so are $\conv_\S(\M)$ and $\conc_\S(\M)$.
\end{lem}

\begin{proof}
One can easily show that $(\conv_\S(\M))_\sigma$ is Jensen convex and $(\conv_\S(\M))_\sigma\le \M_\sigma=\M$. This implies $(\conv_\S(\M))_\sigma  \le \conv_\S(\M)$. 
 
On the other hand, $\M$ is symmetric with respect to the inverse permutation $\sigma^{-1}$ as well. Thus $(\conv_\S(\M))_{\sigma^{-1}} \le \conv_\S(\M)$, which directly implies the inequality $\conv_\S(\M) \le (\conv_\S(\M))_\sigma$. Therefore $\conv_\S(\M)=(\conv_\S(\M))_\sigma$, and $\conv_\S(\M)$ is symmetric with respect to $\sigma$. The proof for $\conc_\S(\M)$ is analogous.
 \end{proof}

\subsection{Reflected means}
Let us recall the notion of reflected means from the paper \cite{PalPas18b}. Let $I$ be and interval and $\M \colon I^n \to I$ be a mean. We define the \emph{reflected mean} of $\M$ as the function $\widehat\M \colon (-I)^n \to (-I)$ given by 
\Eq{*}{
\widehat\M(x_1,\dots,x_n)=-\M(-x_1,\dots,-x_n).
}
Then it is easy to check that $\M$ is convex (or concave) if and only if $\widehat\M$ is concave (or convex), respectively. Moreover $\M \le \Nm$ if and only if $\widehat\M \ge \widehat\Nm$.

For a family of means $\S$, we define its reflected family by $\widehat\S:=\{\widehat\M \colon \M \in \S\}$. 
Let us emphasize that for most of families of means such reflection preserves the family (but reflects the interval). 
For example reflected quasiarithmetic means on $I$ are exactly quasiarithmetic means on $-I$. The same is valid for deviation and quasideviation means. This duality swaps Jensen concave and Jensen convex envelopes. 

\begin{lem}\label{lem:reflected}
 Let $\S$ be a family of means on $I$ and $\M \in \S$. Then $\widehat{\conc_\S(\M)}=\conv_{\widehat\S}(\widehat\M)$.
\end{lem}

\begin{proof}
 Indeed, for fixed $n \in \N$ and $x \in I^n$ one has 
 \Eq{*}{
 \widehat{\conc_\S(\M)}(x)&=-\conc_\S(\M)(-x)\\
 &=-\inf \{\Nm(-x) \colon \Nm \in \S,\, \Nm \text{ is Jensen concave and }\Nm \ge \M \}\\
 &=\sup \{\widehat\Nm(x) \colon \Nm \in \S,\, \Nm \text{ is Jensen concave and }\Nm \ge \M \}\\
 &=\sup \{\widehat\Nm(x) \colon \widehat\Nm \in \widehat\S,\, \widehat\Nm \text{ is Jensen convex and }\widehat\Nm \le \widehat\M \}=\conv_{\widehat\S}(\widehat\M)(x),
 }
 which ends the proof.
\end{proof}

\section{Quasiarithmetic means\label{sec:qa}}

In the next result we establish a complete characterization of the convexity of quasiarithmetic means.
 
 \begin{thm}\label{thm:QAconvex}
 Let $I$ be an open interval, $f \colon I \to \R$ be a continuous, strictly monotone function. Then $\QA{f}$ is convex if and only if the following two conditions are valid:
 \begin{enumerate}
  \item $f \in \Cts(I)$;
  \item either $f''$ is nowhere vanishing and $\frac{f'}{f''}$ is positive and concave, or $f'' \equiv 0$.
 \end{enumerate}
 \end{thm}
\begin{proof}
The implication $(\Leftarrow)$ was already proved by the authors in \cite{PalPas18a}. It was also proved that, under the assumption that $\QA{f}$ is convex, (1) implies (2). Therefore the only remaining part is to show that every convex quasiarithmetic mean is generated by a $\Cts$ function.

As $\QA{f}$ is convex, by Theorem~\ref{thm:M-A}, we obtain $\QA{f} \ge \A$. Thus, by Jensen's inequality and the well-known identity $\QA{f}=\QA{-f}$ one may assume without loss of generality that $f$ is strictly increasing and convex. Then $f$ has strictly positive one-sided derivatives $f_+'$ and $f_-'$ at every point of $I$.

Applying some general results concerning quasideviation means (cf. \cite{Pal88a}), we can obtain that $\QA{f}$ is convex if and only if
there exist $a,\,b \colon I^2 \to \R$ such that
\Eq{E:condC}{
f\big( \tfrac{x+y}2\big)-f\big( \tfrac{u+v}2\big)\le a(u,v) (f(x)-f(u)) + b(u,v) (f(y)-f(v))
}
is valid for all $x,y,u,v \in I$.

For $x>u$ and $y=v$ we obtain
\Eq{*}{
\frac{f\big( \frac{x+v}2\big)-f\big( \frac{u+v}2\big)}{x-u} \le 
a(u,v) \cdot \frac{f(x)-f(u)}{x-u}.
}
Upon taking the limit $x \searrow u$, it follows that  
\Eq{*}{
\tfrac12 \cdot f'_+ \big( \tfrac{u+v}2\big) \le a(u,v) f'_+(u).
}
Therefore
\Eq{*}{
\frac{f'_+ \big( \frac{u+v}2\big)}{2f'_+(u)}\leq a(u,v).
}
Analogously, we obtain
\Eq{*}{
a(u,v) \le \frac{f'_- \big( \frac{u+v}2\big)}{2f'_-(u)},
}
which implies the double inequality
\Eq{E:DI}{
\frac{f'_+ \big( \frac{u+v}2\big)}{2f'_+(u)} \le a(u,v) \le \frac{f'_- \big( \frac{u+v}2\big)}{2f'_-(u)}.
}

Take $p \in I$ arbitrarily. As $f$ is convex, we know that $f'_-(p) \le f'_+(p)$ and $f$ is differentiable everywhere except at countably many points. In particular one can take $u_p \in I$ such that $f$ is differentiable at $u_p$ and $v_p:=2p-u_p \in I$. Then \eq{E:DI} with $(u,v):=(u_p,v_p)$ simplifies to 
\Eq{*}{
\frac{f'_+(p)}{2f'(u_p)} \le \frac{f'_-(p)}{2f'(u_p)}
}
which implies $f'_+(p) \le f'_-(p)$. Consequently, $f$ is differentiable at $p$. As $f$ is convex, we get $f \in \mathcal{C}^{1\#}(I)$ (in particular $f'$, is positive). Then, in view of \eq{E:DI} and the similar inequality for the function $b$, one gets
\Eq{*}{
a(u,v) =\frac{f'\big( \frac{u+v}2\big)}{2f'(u)}, \qquad b(u,v)=\frac{f'\big( \frac{u+v}2\big)}{2f'(v)} \qquad (u,\, v \in I).
}
Now condition \eq{E:condC} can be equivalently rewritten as 
\Eq{E:condC1}{
\frac{f\big( \tfrac{x+y}2\big)-f\big( \tfrac{u+v}2\big)}{f'\big( \frac{u+v}2\big)}\le \frac12 \bigg( \frac{f(x)-f(u)}{f'(u)} + \frac{f(y)-f(v)}{f'(v)}\bigg)\qquad (x,y,u,v \in I).
}
Which implies that the two-variable continuous function $F \colon I^2 \to \R$ given by
\Eq{*}{
F(x,u):=\frac{f(x)-f(u)}{f'(u)}
}
is convex on $I^2$. In particular, for all fixed $x\in I$, the mapping $u\mapsto F_x(u):=F(x,u)$ is convex on $I$. Consequently, $F_x$ is differentiable at every point of $I$ from the left and from the right. However, as $f \in \mathcal{C}^{1\#}(I)$, the mapping $u \mapsto f(x)-f(u)$ is differentiable. Therefore $f'$ is differentiable at every point of $I$ both from the left and from the right. We will denote its one-sided derivatives as $f''_-$ and $f''_+$.

Let $x\in I$ be fixed. By the convexity of $F_x$, for all $v \in I$, there exist a real number $p(x,v)$ such that 
\Eq{*}{
\frac{f(x)-f(u)}{f'(u)} - \frac{f(x)-f(v)}{f'(v)} \ge p(x,v)\cdot(u-v) \qquad\text{for all }u,\,v \in I.
}
Then, for all $u,\,v \in I$,
\Eq{*}{
\frac{(f(x)-f(v))(f'(v)-f'(u))-f'(v)(f(u)-f(v))}{f'(u)f'(v)} \ge p(x,v)\cdot (u-v) 
}
Now assume that $u>v$ and divide by $u-v$ side-by-side. Then we get
\Eq{*}{
\frac{(f(x)-f(v))\frac{f'(v)-f'(u)}{u-v}-f'(v)\frac{f(u)-f(v)}{u-v}}{f'(u)f'(v)} \ge p(x,v).
}
By taking limit $u \searrow v$, we obtain
\Eq{*}{
\frac{(f(v)-f(x))f''_+(v)-f'(v)^2}{f'(v)^2} \ge p(x,v).
}
Repeating the same argumentation for $u<v$, we similarly obtain 
\Eq{*}{
\frac{(f(v)-f(x))f''_-(v)-f'(v)^2}{f'(v)^2} \le p(x,v).
}
The above inequalities then imply that
\Eq{*}{
  (f(v)-f(x))(f''_+(v)-f''_-(v))\ge0 \qquad\text{for all }x,\,v \in I.
}
By taking $x$ to be smaller and bigger than $v$, it follows that $f''_+(v)=f''_-(v)$, which proves the differentiability of $f'$ at $v$.

The remaining part is to show that $f''$ is continuous. However as $F_x$ is convex and differentiable we know that $F_x'$ is continuous and the continuity of $f''$ is straightforward.
\end{proof}

\begin{thm}
Let $I$ be an interval, $f \in \Cts(I)$ be a strictly increasing and convex function. Then $\conv_{\calQ(I)}(\QA{f})=\QA{g}$ for some $g \in \Cts(I)$. 

Moreover either $g''\equiv 0$ (and $\QA{g}$ is the arithmetic mean) or $g''$ is nowhere vanishing and $\frac{g'}{g''}=\conc_{\mathcal{C}(I)}\big(\frac{f'}{f''}\big)$.
\end{thm}

\begin{proof}
Let $P$ be the family of all strictly monotone, affine functions on $I$ and
denote 
\Eq{*}{
U:=\{ h \in \mathcal{C}(I) \colon \QA{h}\text{ is Jensen convex and }\QA{h}\le\QA{f}\}.
}
Then, by the previous theorem,
\Eq{*}{
U=\{ h \in \Cts(I) \colon \QA{h}\text{ is Jensen convex and }\QA{h}\le\QA{f}\}.
}

Obviously $P \subseteq U$ as $\A \le \QA{f}$ and the arithmetic mean is convex. Moreover, by the definition,
\Eq{*}{
\conv_{\calQ(I)}(\QA{f})(x)
&=\sup \{\QA{h}(x) \colon h \in U\} \qquad \text{ for all }x \in \bigcup_{n=1}^\infty I^n.
}
If $U=P$, then obviously $\QA{h}=\A$ for all $h \in U$ and $\conv_{\calQ(I)}(\QA{f})=\A$.
From now on assume that the set $U_0:=U \setminus P$ is nonempty. Then for every $h \in U_0$ we have $\QA{h} \ge \A$ and $\QA{h} \neq \A$. In particular,
\Eq{IneqU0}{
\conv_{\calQ(I)}(\QA{f})(x)
=\sup \{\QA{h}(x) \colon h \in U_0\} \qquad \text{ for all }x \in \bigcup_{n=1}^\infty I^n.
}

By virtue of Theorem~\ref{thm:QAconvex}, for all $h \in U_0$, we have $h \in \Cts(I)$, $h''$ is nowhere vanishing and $\frac{h''}{h'}$ is positive and concave. Moreover, applying well-known comparability criterion, we have $\frac{h''}{h'} \le \frac{f''}{f'}$ for all $h \in U_0$. In particular $\frac{f''}{f'}$ is positive on its domain.

On the other hand, it is relatively easy to verify each function $h\colon I \to \R$ satisfying all properties above belongs to $U_0$. Therefore
\Eq{*}{
U_0=\big\{h \in \Cts(I) \colon h''\text{ is nonvanishing, }\tfrac{h'}{h''}\text{ is concave, and }\tfrac{h'}{h''} \ge \tfrac{f'}{f''}\big\}.
}
Now define $m\colon I \to \R$ by $m:=\conc_{\mathcal{C}(I)}\big(\frac{f'}{f''}\big)$. By $m\ge \frac{f'}{f''}$ we know that $m$ is  positive. Thus the 2nd-order linear ordinary differential equation $\frac{g'}{g''}=m$ has a solution $g \in \Cts(I)$.

Obviously $g''$ is nowhere vanishing and $\frac{g'}{g''}$ is positive and concave. Thus Theorem~\ref{thm:QAconvex} implies that $\QA{g}$ is convex.

On the other hand, by the definition of the concave envelope for every $h \in U_0$, we have $\frac{h'}{h''} \ge \frac{g'}{g''}$ therefore
$\QA{h}\le \QA{g}$. Applying this inequality to all $h \in U_0$ in view of \eq{IneqU0} one gets $\conv_{\calQ(I)}(\QA{f})\le \QA{g}$. 

To verify the converse inequality, observe that $\frac{g''}{g'}=\frac1m\le \frac{f''}{f'}$ which implies $\QA{g}\le \QA{f}$. Thus $\QA{g}$ is a convex minorant of $\QA{f}$, equivalently $g \in U_0$. Applying the inequality \eq{IneqU0} we obtain $\conv_{\calQ(I)}(\QA{f})\ge \QA{g}$. 
\end{proof}

Using Lemma~\ref{lem:reflected} we can formulate the result concerning concave envelopes in a family of quasiarithmetic means. 

\begin{cor}
Let $I$ be an interval, $f \in \Cts(I)$ be an increasing and concave function. Then $\conc_{\calQ(I)}(\QA{f})=\QA{g}$ for some $g \in \Cts(I)$. 

Moreover either $g''\equiv 0$ (and $\QA{g}$ is the arithmetic mean) or $g''$ is nowhere vanishing and $\frac{g'}{g''}=\conv_{\mathcal{C}(I)}\big(\frac{f'}{f''}\big)$.
\end{cor}


\end{document}